\documentclass[graybox]{svmult}

\usepackage{mathptmx}       % selects Times Roman as basic font
\usepackage{helvet}         % selects Helvetica as sans-serif font
\usepackage{courier}        % selects Courier as typewriter font
\usepackage{makeidx}         % allows index generation
\usepackage{graphicx}        % standard LaTeX graphics tool
\usepackage{multicol}        % used for the two-column index
\usepackage[bottom]{footmisc}% places footnotes at page bottom
\usepackage{amsmath,amssymb,amsfonts}        % AMS math packages
\newcommand{\halmos}{\hfill{\ensuremath\blacksquare}}

\newtheorem{qstn}{Question}

\def\eref#1{$(\ref{#1})$}

\def\tref#1{Theorem~$\ref{#1}$}
\def\cref#1{Conjecture~$\ref{#1}$}
\def\qref#1{Question~$\ref{#1}$}

\renewcommand{\geq}{\geqslant}
\renewcommand{\leq}{\leqslant}
\renewcommand{\ge}{\geqslant}
\renewcommand{\le}{\leqslant}
\renewcommand\emptyset{\varnothing}

\newcommand{\eps}{\varepsilon}

\def\dfrac#1#2{\lower0.15ex\hbox{\large$\frac{#1}{#2}$}} 

\DeclareMathOperator{\LHC}{LHC}
\DeclareMathOperator{\NC}{NC}
\DeclareMathOperator{\NE}{NE}
\newcommand{\lhc}[3]{\ensuremath{\LHC_{#1}(#2,#3)}}
\newcommand{\ncom}[2]{\ensuremath{\NC_{#1}(#2)}}
\newcommand{\next}[2]{\ensuremath{\NE_{#1}(#2)}}

\title*{Extendibility of Latin Hypercuboids}

\author{Candida Bowtell\thanks{CB gratefully acknowledges support from Leverhulme Grant ECF--2023--393 and ERC Starting Grant 947978.}, 
Alice Devillers\thanks{AD was supported by Australian Research Council Discovery Project DP200100080.},
Andr\'e K\"undgen, Padraig \'{O} Cath\'{a}in\thanks{P\'{O}C acknowledges support from Technical University of the Shannon through a \textit{Learning Enhancement Project}, and colleagues in DCU for facilitating participation in the MATRIX workshop at which this work was initiated.} and Ian M. Wanless}

\institute{C. Bowtell \at School of Mathematics, University of Birmingham, Birmingham, B15 2TT, UK. \email{c.bowtell@bham.ac.uk}
\and A. Devillers\at Department of Mathematics and Statistics, University of Western Australia, Perth, Australia \email{alice.devillers@uwa.edu.au}
\and A.~K\"undgen\at Department of Mathematics, California State University San Marcos, San Marcos, CA 92096, United States. \email{akundgen@csusm.edu}
\and
P. \'{O} Cath\'{a}in\at Fiontar agus Scoil na Gaeilge, Dublin City University, Dublin 9, Ireland.
\email{p.ocathain@gmail.com}
\and
I.M.Wanless\at
School of Mathematics, Monash University, Clayton, Vic 3800, Australia.
\email{ian.wanless@monash.edu}
}

\begin{document}

\maketitle

\abstract{
A Latin hypercuboid of order $n$ is a $d$-dimensional matrix of
dimensions $n\times n\times\cdots\times n\times k$, with symbols from
a set of cardinality $n$ such that each symbol occurs at most once in
each axis-parallel line.  If $k=n$ the hypercuboid is a Latin
hypercube. The Latin hypercuboid is \emph{completable} if it is
contained in a Latin hypercube of the same order and dimension. It is
\emph{extendible} if it can have one extra layer added. In this note
we consider which Latin hypercuboids are completable/extendible. We also consider a generalisation that involves multidimensional arrays of sets that satisfy certain balance properties. The extendibility problem corresponds to choosing representatives from the sets in a way that is 
analogous to a choice of a Hall system of distinct representatives, but in higher dimensions. The completability problem corresponds to partitioning the sets into such SDRs. We provide a construction for such an array of sets that does not have the property analogous to completability. A related concept was introduced by H\"aggkvist under the name $(m,m,m)$-array. We generalise a construction of $(m,m,m)$-arrays credited to Pebody, but show that it cannot be used to build the arrays that we need.}

\section{Introduction}

Let \lhc{d}{n}{k} be the set of $d$-dimensional matrices of dimensions
$n\times n\times\cdots\times n\times k$, with symbols from the set
$[n]=\{1,2,\dots,n\}$ such that each symbol from $[n]$ occurs at most
once in each line that is parallel to one of the axes.  Such matrices
are called \emph{Latin hypercuboids} of order $n$ and dimension $d$.
If $k=n$ then the Latin hypercuboid is a Latin hypercube. A
hypercuboid in \lhc{d}{n}{k} is \emph{extendible} if it is contained
in some hypercuboid in \lhc{d}{n}{k+1} and \emph{completable} if it is
contained in some Latin hypercube of order $n$ and dimension $d$. For $H\in\lhc{d}{n}{k}$, we denote the entry in cell $[i_1,\dots,i_d]$ of $H$ by $H[i_1,\dots,i_d]$.
We will use
$$U_H(i_1,\dots,i_{d-1})=[n]\setminus\{H[i_1,\dots,i_{d-1},i]:i\in\{1,\dots,k\}\}$$
to denote the set of as yet unused symbols in a particular line of $H$. So $H$ is extendible if and only if we can
choose
$u(i_1,\dots,i_{d-1})\in U_H(i_1,\dots,i_{d-1})$ for each $(i_1,\dots,i_{d-1})$
in such a way that $u(i_1,\dots,i_{d-1})\ne u(j_1,\dots,j_{d-1})$ whenever
the vectors
$(i_1,\dots,i_{d-1})$ and $(j_1,\dots,j_{d-1})$ agree in all but one coordinate.

In the case when $d=2$, Latin hypercuboids and hypercubes are better
known as \emph{Latin rectangles} and \emph{Latin squares},
respectively.  When $d=3$, Latin hypercuboids and hypercubes are
sometimes called Latin cuboids and Latin cubes, respectively. The term
\emph{Latin parallelepiped} was also used in the early literature on
the subject but it does not capture the geometrical essence of the
objects (nobody would call a Latin rectangle a Latin parallelogram!).
A famous theorem due to Marshall Hall \cite{Hal45} says that every
Latin rectangle is completable (to a Latin square). In all higher
dimensions, there exist noncompletable Latin hypercuboids. This
immediately raises many natural questions, about how thin
nonextendible hypercuboids can be (i.e.~how small $k$ can be relative
to $n$), whether hypercuboids typically are or are not extendible and
so on. Such questions are the subject of this article. 

Let $\next{d}{n}$ be the smallest
$k$ such that there exists a hypercuboid in \lhc{d}{n}{k}
that is \emph{not} extendible.
Let $\ncom{d}{n}$ be the smallest
$k$ such that there exists a hypercuboid in \lhc{d}{n}{k}
that is \emph{not} completable. Of course, $\ncom{d}{n}\le\next{d}{n}$ for
all $d,n$. Note also the simple observation that every hypercuboid in \lhc{d}n{n-1} is completable (and thus extendible). 
It follows that the fullest
noncompletable hypercuboids that might exist are in \lhc{d}n{n-2}. So for each $d$ and $n$, either we have $\ncom{d}{n} \leq n-2$ or for every $k \in [n-1]$ and $H \in \lhc{d}n{k}$ we have that $H$ is completable. (Of course it then also follows that every such $H$ is extendible.) 

In this note, we survey the known upper and lower bounds for $\next{d}{n}$ and $\ncom{d}{n}$. We also generalise the extendibility and completability problems to objects that we call $(n^d,k)$-arrays. We give both negative and positive results regarding constructions for these arrays and relate them to an existing notion of an $(m,m,m)$-array introduced by H\"aggkvist \cite{Hag89}. 
We conclude by discussing a variety of open questions in the area.

\section{Arrays of sets}

Define an $(n^d,k)$-array to be a $d$-dimensional array with the following properties:
\begin{itemize}
\item the dimensions of the array are $n\times n\times\cdots\times n$,
\item each entry of the array is a subset of $[n]$ of cardinality exactly $k$, 
\item every number in $[n]$ occurs exactly $k$ times within the entries along any axis-parallel line within the array.
\end{itemize}

It is immediate from this definition that the array $U_H$ of unused symbols is an $(n^{d-1},n-k)$-array for any $H\in\lhc dnk$. We say that an  $(n^{d-1},n-k)$-array $A$ is \emph{realisable} if $A=U_H$ for some $H\in\lhc dnk$. However, there are values of $d,n,k$ for which not all $(n^{d-1},n-k)$-arrays are realisable. To see this, suppose that there is a noncompletable $H'\in\lhc dn{n-k}$ and consider the $(n^{d-1},n-k)$-array $(i_1,\dots,i_{d-1})\rightarrow[n]\setminus U_{H'}(i_1,\dots,i_{d-1})$. This array cannot be $U_H$ for any $H\in\lhc dnk$, otherwise $H'$ would be completable.

To mirror the properties of Latin hypercuboids that we are interested in, we say that an $(n^d,k)$-array contains a \emph{layer} if it is possible to replace each set in the array by a representative of that set, in such a way that the result is in $\lhc dn1$. We say that the array is \emph{layerable} if it can be decomposed into $k$ layers. Note that $U_H$ contains a layer if and only if $H$ is extendible and $U_H$ is layerable if and only if $H$ is completable. It seems of interest to study these properties for $(n^d,k)$-arrays without needing to worry about whether the array is realisable. We take a modest first step in that direction with \tref{t:nonlayerable} below.

There is some relevant literature that generalises $(n^2,k)$-arrays by replacing ``exactly" with ``at most" both times it occurs in our definition.
H\"aggkvist~\cite{Hag89} defined an object, called an $(m,m,m)$-array, to be an $n\times n$
array of sets of size at most $m$, with the property that each number
in $[n]$ occurs at most $m$ times among the sets in any row or column
of the array. 
Such an array $A$ is said to be \emph{avoidable} if
there exists an $n\times n$ Latin square $L$ such that no entry of $L$
is contained within the set in the corresponding cell of $A$. 
The sets $[n]\setminus U_H(i,j)$ of ``used'' symbols in
$H\in\lhc3n{m}$ provide an $(m,m,m)$-array $M_H$. Any Latin square avoiding
$M_H$ can be used to extend $H$, and is a layer in the $(n^2,n-m)$-array formed by taking the complement of each set in $M_H$. H\"aggkvist \cite{Hag89} made this conjecture:

\begin{conjecture}\label{cj:Hag}
  There exists a constant $\gamma>0$ such that if $m<\gamma n$ then
  every $(m,m,m)$-array is avoidable.
\end{conjecture}

If this conjecture is true, then $\gamma\le1/3$, as shown by the next result, which generalises a construction attributed in \cite{CO06} to Pebody.

\begin{theorem}\label{t:unavoidable}
    Let $n=a+b+c$ and $m=\max(a,b,c)$ where $a,b,c$ are positive integers that are not all equal. Let $A=\{1,\dots,a\}$, $B=\{a+1,\dots,a+b\}$ and $C=\{a+b+1,\dots,n\}$. Define $M=(M_{i,j})$ to be an $(m,m,m)$-array in which 
    \begin{equation}\label{e:ABC}    
    M_{i,j}=\begin{cases}
        A&\text{if $i\in A$ and $j\in A$},\\
        B&\text{if $i\in B$ and $j\in B$},\\
        C&\text{if $i\in C$ and $j\in C$},\\
\emptyset&\text{otherwise},
    \end{cases}
    \end{equation}
   as illustrated in the following diagram:
\[
    \begin{array}{|c|c|c|}
    \hline
    A&&\\
    \hline
    &B&\\
    \hline
    &&C\\
    \hline
    \end{array}.
    \]
    Further suppose that $N=(N_{i,j})$ is an $(n^2,k)$-array for some $k\ge 1$.
    Then there exists some $i,j$ for which $M_{i,j}\cap N_{i,j}\neq\emptyset$.
\end{theorem}

\begin{proof}
Aiming for a contradiction,    suppose that there is an $(n^2,k)$-array $N=(N_{i,j})$ for some $k\ge 1$ such that $M_{i,j}\cap N_{i,j}=\emptyset$ for all $i,j\in[n]$. Let $t=\sum|N_{i,j}|$, where the sum is over all $(i,j)\in(A\times A)\cup(B\times B)\cup(C\times C)$. Since $|N_{i,j}|=k$ for all $i,j$, we know that $t=k(a^2+b^2+c^2)$. However, we will argue that there are not enough symbols available to achieve this value of $t$.
First, consider the symbols in $A$. Each of these $a$ symbols must occur $k$ times in each row (and column) for a total of $nk$ occurrences. Similarly, each of them occurs $ka$ times in the first $a$ rows of $N$ and also $ka$ times in the first $a$ columns of $N$. However, they are forbidden in the intersection of these rows and columns. It follows that each of the symbols in $A$ occurs exactly $nk-2ak$ outside of the first $a$ rows and columns. So the contribution to $t$ from symbols in $A$ is at most $a(nk-2ak)$. By analogous arguments the contributions to $t$ from symbols in $B$ and $C$ are at most $b(nk-2bk)$ and $c(nk-2ck)$, respectively. So 
\[
k(a^2+b^2+c^2)=t\le
a(nk-2ak)+b(nk-2bk)+c(nk-2ck).
\]
Dividing both sides by $k$ and rearranging, we see that
\[
3(a^2+b^2+c^2)\le
(a+b+c)n=(a+b+c)^2.
\]
The arithmetic mean - quadratic mean inequality now implies that $a=b=c$, completing the proof. \halmos
\end{proof}

Note that the $k=1$ case of \tref{t:unavoidable} says precisely that $M$ as defined in \eref{e:ABC} is an unavoidable $(m,m,m)$-array. The statement for more general $k$
eliminates any hope of embedding $M$ in an $(n^2,n-k)$-array with $k>0$. If such an embedding had been possible then, by complementing, we would have obtained an $(n^2,k)$-array that contains no layer.

\section{Latin cuboids}\label{s:cubes}

We start our discussion of hypercuboids by considering the case $d=3$, which, perhaps unsurprisingly, has been studied much more than any $d>3$. The
first result, from 1982, was by Hor\'ak \cite{Hor82}, who proved that $\ncom3{2^k}\le2^k-2$.
Then in 1986, Fu \cite{Fu86} proved that $\ncom3n\le n-2$ when $n=6$ or $n\ge12$. Completing the picture for when $\ncom3n\le n-2$, Kochol \cite{Koc89} showed this:

\begin{theorem}\label{t:Koch}
  $\ncom3n\le n-2$ if and only if $n\ge5$.
\end{theorem}

Looking for tighter upper bounds, in \cite{Koc91}, the same author proved that $\ncom3n\le n-c$
for $c\ge3$ whenever $n\in\{3c,4c,5c\}$ or $n\ge6c$.
Then, in \cite{Koc95}, Kochol proved that for any $k$ and $n$ satisfying
$\frac{1}{2}n<k\le n-2$ there is a noncompletable Latin cuboid in
\lhc3nk. In particular, this gives for every $n\ge5$ that $\ncom3n \leq \lfloor{n/2}\rfloor +1$. Furthermore, it is simple to use such examples to create noncompletable
Latin hypercuboids in \lhc{d}nk for all $d\ge3$ (see Theorem \ref{t:monotone} in Section~\ref{sec_hd}).  Kochol conjectured
that all noncompletable Latin cuboids are more than half-full, but
examples of noncompletable $5\times 5\times 2$, $6\times 6\times 2$,
$7\times 7\times 3$ and $8\times8\times4$ Latin cuboids were
subsequently given in \cite{MW08}.

Kochol \cite{Koc12} showed that $\next3n\le n-m$
for even $m>2$ and $n\ge4m-2$.
The same paper contains the following useful general embedding result:

\begin{theorem}\label{t:embeddbl}
    If there exists a nonextendible (respectively, noncompletable) cuboid in
\lhc3r{r-m} then there exists a nonextendible (respectively, noncompletable)
cuboid in \lhc3n{n-m} for every $n\ge2r$. 
\end{theorem}

Finally, Bryant et al. \cite{BCMPW12} showed the following two
results:

\begin{theorem}\label{t:exactly-half}
$\ncom3{2m}\le m$ for all $m \geq 4$.
\end{theorem}

\begin{theorem}\label{t:under-half}
  $\next3{2m-1}\le m-1$ for all even $m\notin\{2,6\}$.
\end{theorem}

The case
$m=2$ is a genuine exception in \tref{t:under-half}, in the sense that
every Latin cuboid in \lhc331 is completable.  It is unknown whether
$m=6$ is a genuine exception.  Together with the results  from \cite{Koc95,MW08} mentioned above, \tref{t:exactly-half} shows:

\begin{theorem}\label{t:stateofartcom}
    $\ncom3n\le\lceil n/2\rceil$ for all $n\ge5$.
\end{theorem}

Combining \tref{t:under-half} with \tref{t:embeddbl}, or by using \cite{Koc12}, we have:

\begin{theorem}\label{t:stateofartext}
\[\next3n\le
\begin{cases}
(n-1)/2&\text{if $n\equiv3\bmod4$,}\\
3n/4+O(1)&\text{otherwise}.
\end{cases}
\]
\end{theorem}

\tref{t:stateofartcom} and \tref{t:stateofartext} represent the tightest known upper bounds on $\NC_3$ and $\NE_3$ currently known for general $n$.
It would  be interesting to know
whether an analogue of \tref{t:under-half} holds for odd $m$. In that direction, we can show:

\begin{theorem}\label{t:nonlayerable}
  For every $n\ge5$ there exists a nonlayerable $(n^2,\lfloor n/2\rfloor)$-array.  
\end{theorem}

\begin{proof}
    The construction used to prove \tref{t:exactly-half} suffices if $n$ is even. We will modify that construction for the case when $n=2m+1$ is odd. The key idea is to embed a nonlayerable subarray in the first $m$ rows and $m$ columns. Let $S=\{1,\dots,m\}$ and $T=\{m+1,\dots,n\}$. Define an $(m+1)\times(m+1)$ array $B=(b_{i,j})$, with symbols from $T$, by $b_{i,j}\equiv i+j-3\bmod (m+1)$. Then we define our $(n^2,\lfloor n/2\rfloor)$-array $A=(A_{i,j})$ by
    \begin{align*}
        A_{ij}=\begin{cases}
        S\cup\{n\}\setminus\{1\}&\text{if $i=1$ and $j\le m$},\\
        S\cup\{n\}\setminus\{2\}&\text{if $2\le i\le m$ and $j=1$},\\%[1ex]
      S&\text{if $2\le i\le m$ and $2\le j\le m$},\\[1ex]
      T\cup\{1\}\setminus\{n,b_{1,j-m}\} &\text{if $i=1$ and $j\ge m+2$},\\
        T\cup\{2\}\setminus\{n,b_{i,1}\} &\text{if $2\le i\le m$ and $j=m+1$},\\
       T\setminus\{b_{i,j-m}\} &\text{if $(i,j)=(1,m+1)$ or ($2\le i\le m$ and $j>m+1$)},\\[1ex]
    T\cup\{1\}\setminus\{n,b_{i-m,j}\} &\text{if $i+j=n$ and $j\le m$},\\
    T\cup\{2\}\setminus\{n,b_{i-m,j}\} &\text{if $i\in T\setminus\{m+1,n-1\}$ and $j=1$},\\
    T\setminus\{b_{i-m,j}\} &\text{if $(i,j)=(m+1,1)$ or ($2\le j\le m<i$ and $i+j\ne n$)},\\[1ex]
      S\cup\{n\}\setminus\{1\}&\text{if $m<i=j-1$},\\
      S\cup\{n\}\setminus\{2\}&\text{if $i\in T\setminus\{m+1,n-1\}$ and $j=m+1$},\\
        T\setminus\{b_{m+1,j-m}\} &\text{if $m<i=j$},\\
        S&\text{otherwise}.
        \end{cases}
    \end{align*}
It is routine to check that $A$ is an $(n^2,\lfloor n/2\rfloor)$-array. 
We next argue that it is not layerable. It suffices to show that there is no layer $L=(l_{i,j})$ in $A$ that has $l_{1,1}=n$. Let $M$ be the submatrix at the intersection of the first $m$ rows and first $m$ columns of such an $L$. The only symbols that can appear in $M$ are $S\cup\{n\}$. However, $n$ cannot occur more than once, and $1$ and $2$ cannot appear in row $1$ or column $1$ respectively. Each symbol in $S\setminus\{1,2\}$ can appear $m$ times in $M$. It follows that $M$ cannot be filled, so $L$ does not exist. \halmos
\end{proof}

Note that we can do slightly better than \tref{t:nonlayerable} when $n\equiv3\bmod4$, because in that case \tref{t:under-half} gives us an $(n^2,\lceil n/2\rceil)$-array that does not contain any layer.

Various authors working on $(m,m,m)$-arrays have proved results that include extendibility of Latin cuboids as a special case.
Cutler and \"Ohman \cite{CO06} showed for all $m$ that every
Latin cuboid in \lhc3{2mk}m is extendible, provided $k$ is
sufficiently large.
Andr\'en's PhD thesis \cite{And10} proved \cref{cj:Hag} for even $n$,
although it seems she did not produce a journal version of this result.
Instead, she published a proof \cite{And12} only for the case when
$n$ is a power of $2$. She then teamed up with Casselgren and \"Ohman
\cite{ACO13}, to publish a proof of the odd case of \cref{cj:Hag}.

\medskip

We conclude this section by considering a graphical reformulation of our setting of Latin cuboids and $(n^2,k)$-arrays. There is a natural bijection between Latin squares of order $n$ and optimally coloured complete bipartite graphs. 
In particular, a Latin square of order $n$ using symbols $\{1, \ldots, n\}$ corresponds to a complete bipartite graph $K_{n,n}$ where we think of one set of the vertices as the rows $1, \ldots, n$ of the Latin square, the other set of vertices as the columns $1, \ldots, n$ of the Latin square, and we colour the edge $ij \in E(K_{n,n})$ with the colour $c(ij) \in \{1, \ldots, n\}$ which corresponds to the symbol in row $i$ and column $j$ of the Latin square. Now, for the case of Latin cuboids (i.e.\ when $d=3$), and more generally $(n^2, k)$-arrays, it may prove useful to consider the problems of extendibility and the existence of a layer, respectively, as a list colouring problem in $K_{n,n}$. This can be realised in the following way:

\begin{qstn}
    Let $K_{n,n}$ be such that for each edge $e \in E(K_{n,n})$ there exists a list $L_e$ of exactly $k$ colours in $\{1, \ldots, n\}$ such that for every $i \in \{1, \ldots, n\}$ and every vertex $v \in K_{n,n}$ there exist exactly $k$ edges $e$ incident to $v$ such that $i \in L_e$. For which values of $k$ does every such collection of lists admit a proper colouring of $K_{n,n}$ from these lists?
\end{qstn}

Note that a proper colouring of $K_{n,n}$ in this case gives rise to an optimally coloured $K_{n,n}$, or equivalently, a Latin square. The lists $\{L_e\}_{e \in E(K_{n,n})}$ correspond to an $(n^2,k)$-array, and whether or not $K_{n,n}$ is properly colourable from these lists corresponds to whether the $(n^2,k)$-array contains a layer. This may prove to be an insightful way to study the existence of layers in arrays and extendibility of Latin cuboids.

\bigskip

\section{Higher dimensions}\label{sec_hd}

We now discuss what is known for dimensions $d>3$. We start with an easy observation that $\NC$ and $\NE$ decrease monotonically as dimension increases.

\begin{theorem}\label{t:monotone}
  If $d'>d$ then $\ncom{d'}n\le\ncom{d}n$ and $\next{d'}n\le\next{d}n$.
\end{theorem}

\begin{proof}
  Suppose that $H\in\lhc{d}{n}{k}$. Then we can define $I\in\lhc{d'}nk$ by
  $$I[i_1,\dots,i_{d'}]\equiv i_{1}+\dots+i_{d'-d}+H[i_{d'-d+1},\dots,i_{d'}]\bmod n.$$
  If $I$ is extendible to $I'\in\lhc{d'}n{k'}$ for some $k'>k$
  then define $H'\in\lhc{d}n{k'}$ by
  $$H'[i_{d'-d+1},\dots,i_{d'}]=I'[n,n,\dots,n,i_{d'-d+1},\dots,i_{d'}]$$
  and note that $H'$ is an extension of $H$. Hence $I$ is
  noncompletable (respectively nonextendible) if $H$ is noncompletable
  (respectively nonextendible). The result follows. \halmos
\end{proof}

Potapov \cite{Pot12} shows that every \lhc{d}{n}{k} is completable
for $k<n\le4$ and arbitrary $d$. It then follows from \tref{t:monotone}
and \tref{t:Koch} that $\ncom{d}{n}$ and $\next{d}{n}$ are defined
if and only if $n\ge5$ and $d\ge3$.

A result of Casselgren, Markstr\"om and Pham \cite{CMP19} implies that
there is a constant $\gamma>0$ such that every hypercuboid in
\lhc{4}n{\gamma n} is extendible provided $n$ is a power of 2. This
was soon generalised to all even $n$ by Casselgren and Pham
\cite{CP20}.

\begin{theorem}\label{t:4dim}
  There is a constant $\gamma>0$ such that $\next4n\ge\gamma n$ for
  all even $n$.
\end{theorem}

Casselgren and Pham conjectured the corresponding result holds also
for odd $n$. There is nothing shown for higher dimensions.  In light
of \tref{t:monotone}, we note that \tref{t:4dim} implies the
corresponding result for three dimensions, first shown in \cite{And10}.
We thus have a fully refereed proof of \cref{cj:Hag}.

\section{Open questions}

As the literature review in the Section~\ref{s:cubes} demonstrates, there
is still much we do not know about extendibility and completability of
Latin hypercuboids. By the result of Potapov mentioned above, we may as well
assume henceforth that $n\ge5$.

\subsection{Extendibility}

\begin{qstn}\label{q:cd}
  For each dimension $d\ge3$, does there exist a constant $c_d>0$ such that
  $\next{d}n=c_dn+o(n)$?
\end{qstn}

Even in three dimensions the true growth rate is unclear. From the proof of
\cref{cj:Hag} we know that there is a linear lower bound on $\next3n$
(with an unspecified constant), and we know that, if $c_3$ exists then
$c_3\le1/2$ from \tref{t:under-half}. 
As shown in \tref{t:stateofartext},
we have a much stronger result for $n\equiv3\bmod4$ than we have for other $n$. It is natural to try to do at least as well  when $n\not\equiv3\bmod4$.

\begin{qstn}\label{q:halfn}
  Is $\next3{n}\le n/2$ for all large $n$?
\end{qstn}

In four dimensions we have a linear lower bound on $\next4n$ from
\tref{t:4dim}, but only for even $n$. It must surely be true, as
Casselgren and Pham \cite{CP20} conjecture, that ``even'' can be
dropped.

\begin{qstn}\label{q:linlow}
  Is there a constant $\gamma>0$ such that $\next4n\ge\gamma n$ for
  all $n\ge5$?
\end{qstn}

For dimensions higher than 4, the situation is wide open. It may be that
\qref{q:cd} is inaccessible in the near future. A more modest goal
is to settle this:

\begin{qstn}
  For $d>4$ does there exist a $\gamma_d>0$ such that $\next{d}{n}\ge\gamma_dn$ for all $n\ge5$?
\end{qstn}

We are unsure what we expect the answer to this question will turn out
to be.  It is not out of the question that $\next{d}{n}=o(n)$ for
large $d$.

\subsection{Completability}

If the situation for extendibility is quite unclear, the picture for
completability is positively opaque. The main thing we know is that $\ncom dn\le\lceil n/2\rceil$ by combining \tref{t:monotone} with \tref{t:stateofartcom}.
%\begin{qstn}  %% Essentially known!
%  Is $\ncom3{n}\le n/2$ for all large $n$?
%\end{qstn}
%Of course, an affirmative answer to \qref{q:halfn} would also settlethis question. 
It is worthwhile to try to improve this bound, even in 3 dimensions, and interesting to consider what happens as the dimension increases.

\begin{qstn}
    Is $\ncom dn=o(1)n$ as $d\rightarrow\infty$, with $n$ fixed?
\end{qstn}

Perhaps even more challenging is the task of proving lower
bounds on $\ncom{d}{n}$. To our knowledge, the only such known bound is the
rather trivial observation that $\ncom{d}{n}\ge2$ for all $n,d$ because
any Latin hypercuboid with a single layer can be developed cyclically to
complete a Latin hypercube. We thus could ask the analogous question
to \qref{q:linlow}, but perhaps it is more appropriate to ask for any kind
of increasing lower bound:

\begin{qstn}
  For fixed $d$, does $\ncom{d}{n}\rightarrow\infty$ as $n\rightarrow\infty$?
\end{qstn}

Another fundamental and interesting line of enquiry is about the
relationship between $\next{d}{n}$ and $\ncom{d}{n}$. We know the
latter is bounded above by the former, but when is this inequality
strict?

\begin{qstn}
  For which $d,n$ is $\next{d}{n}>\ncom{d}{n}$?
\end{qstn}

It is quite possible that the gap between these functions grows large,
but this is not known.

\begin{qstn}
  For each $d$, is $\next{d}{n}-\ncom{d}{n}$ bounded as $n\rightarrow\infty$? Is $\next{d}{n}/\ncom{d}{n}$ bounded?
\end{qstn}

Another line of enquiry is to consider an intermediate zone between
extendibility and completability. Can many extra layers be added?

\begin{qstn}
  For fixed $d$ does there exist a small $\eps>0$ such that every
  hypercuboid in \lhc{d}{n}{\eps n+o(n)} can be extended to a hypercuboid
  in \lhc{d}{n}{(1-\eps)n+o(n)}?
\end{qstn}

Obviously such an extension cannot be done carelessly layer-by-layer,
without risking hitting one of the obstacles embodied in
\tref{t:under-half}.

\subsection{Random Latin hypercuboids}

Another source of open problems is to consider properties of random
Latin hypercuboids. Here we are not aware of any results, so the field is
ripe for exploration.

\begin{qstn}\label{q:randext}
  For which constants $d,c$ is it true that almost all hypercuboids in
  \lhc{d}{n}{cn+o(n)} are extendible?
\end{qstn}

\begin{qstn}
  For which constants $d,c$ is it true that almost all hypercuboids in
  \lhc{d}{n}{cn+o(n)} are completable? 
\end{qstn}

The constructions used to prove results such as \tref{t:under-half} involve
carefully structured sets $U_H(i_1,\dots,i_{d-1})$.
One approach to making progress on \qref{q:randext} would be to show that
if the sets $U_H(i_1,\dots,i_{d-1})$ satisfy some pseudorandomness property then
the hypercuboid $H$ must be extendible.
We say that $H\in\lhc{d}{n}{k}$ is \emph{$\delta$-regular} if 
$$\bigg|\frac{n}{(n-k)^2}\,\big|U_H(i_1,\dots,i_{d-1})\cap U_H(j_1,\dots,j_{d-1})\big|-1\bigg|\le\delta$$
  whenever the vectors
  $(i_1,\dots,i_{d-1})$ and $(j_1,\dots,j_{d-1})$ agree in all but one coordinate.
  
\begin{qstn}
  For which $d,n,k,\delta$ does it follow that every $\delta$-regular
  $H\in\lhc dnk$ must be extendible?
\end{qstn}

\subsection{Arrays of sets}

All of the questions discussed above have their analogues for $(n^d,k)$-arrays. Some of those questions may be easier to approach without the need to consider whether an array is realisable. For example, we may ask:

\begin{qstn} For which constants $d,c$ do almost all $(n^d,cn+o(n))$-arrays contain a layer?
\end{qstn}

\begin{qstn} For which constants $d,c$ are almost all $(n^d,cn+o(n))$-arrays  layerable?
\end{qstn}

It is also interesting to consider whether realisability has a material affect on outcomes.

\begin{qstn}
   For $n\rightarrow\infty$ and fixed $d$, is the smallest $k$ for which there exists an $(n^{d-1},n-k)$-array that contains no layer asymptotically equal to $\next{d}{n}$? 
\end{qstn}

\begin{qstn}
   For $n\rightarrow\infty$ and fixed $d$, is the smallest $k$ for which there exists a nonlayerable $(n^{d-1},n-k)$-array asymptotically equal to $\ncom{d}{n}$? 
\end{qstn}

\begin{acknowledgement}
This work was commenced during the Extremal Problems in Graphs, Designs, and Geometries workshop at the Australian mathematical research institute MATRIX. The authors are grateful to MATRIX and the organisers of this workshop. 
\end{acknowledgement}

\let\oldthebibliography=\thebibliography
  \let\endoldthebibliography=\endthebibliography
  \renewenvironment{thebibliography}[1]{%
    \begin{oldthebibliography}{#1}%
      \setlength{\parskip}{0.4ex plus 0.1ex minus 0.1ex}%
      \setlength{\itemsep}{0.4ex plus 0.1ex minus 0.1ex}%
  }%
  {%
    \end{oldthebibliography}%
  }

\end{document}